\documentclass{amsart}

\usepackage[margin=1.5in]{geometry}  
\usepackage{graphicx}              
\usepackage{amsmath}               
\usepackage{amsfonts}              
\usepackage{amsthm}                
\usepackage{enumitem}
\usepackage{amssymb}
\usepackage{hyperref}
\usepackage{tabularx}
\usepackage{amsmath}
\usepackage{mathtools}

\usepackage[mathscr]{euscript}
\usepackage[utf8]{inputenc}
\usepackage[english]{babel}
\usepackage{multicol}

\usepackage{tikz}
\usetikzlibrary{arrows}
\usetikzlibrary{matrix}

\tikzset{diagram/.style={matrix of math nodes, inner sep=0pt, row
    sep=#1, column sep=2.5em, text height=1.5ex, text depth=.25ex,
    nodes={inner sep=1ex}}}
\tikzset{diagram/.default=2.5em}

\newtheorem{thm}{Theorem}
\newtheorem{prop}[thm]{Proposition}
\newtheorem{lemma}[thm]{Lemma}
\newtheorem{cor}[thm]{Corollary}

\newtheorem*{thm*}{Theorem}

\theoremstyle{definition}
\newtheorem{defn}{Definition}
\newtheorem{ex}{Example}
\newtheorem{rmk}{Remark}
\newtheorem*{claim}{Claim}
\newtheorem*{related works}{Related Works}

\newtheorem*{question*}{Question}

\theoremstyle{definition}

\numberwithin{equation}{section}

\newcommand{\sq}{\mathrm{Sym}^{n}V(\fq)} 
\newcommand{\sv}{\mathrm{Sym}^{n}V} 
\newcommand{\sym}{\mathrm{Sym}}

\newcommand{\RR}{\mathbb{R}}
 
\newcommand{\QQ}{\mathbb{Q}}
\newcommand{\NN}{\mathbb{N}}
\newcommand{\A}{\mathbb{A}}

\newcommand{\ZZ}{\mathbb{Z}}
\newcommand{\E}{\mathbb{E}}

\newcommand{\s}{\sigma} 
\newcommand{\ld}{\lambda}

\newcommand{\fq}{\mathbb{F}_q}

\newcommand{\ff}{\mathbb{F}}

\newcommand{\zt}{Z(V,t)}
\newcommand{\zto}{\mathring{Z}(V,t)}
\newcommand{\zod}{\mathring{Z}(V,q^{-d})}
\newcommand{\CCC}{\mathscr{C}}
\newcommand{\PPP}{\mathscr{P}}
\newcommand{\DDD}{\mathscr{D}}
\newcommand{\pr}{\mathrm{Prob}}

\newcommand{\cvq}{\mathrm{Conf}^{n}V(\fq)}

\newcommand{\cv}{\mathrm{Conf}^{n}V}

\newcommand{\cq}{\cvq}

\newcommand{\conf}{\mathrm{Conf}}

\newcommand{\oo}{\omega}
\newcommand{\OO}{\Omega}
\newcommand{\mn}{\pi_n}
\newcommand{\mk}{\pi_k}
\newcommand{\mj}{\pi_j}

\newcommand{\ep}{\epsilon}

\newcommand{\ee}{\frac{e^{-L}L^r}{r!}}
\newcommand{\mm}{\mu_{r,L}}

\def\mul#1#2{\ensuremath{\left(\kern-.3em\left(\genfrac{}{}{0pt}{}{#1}{#2}\right)\kern-.3em\right)}}

\AtEndDocument{\bigskip{\footnotesize\par
\textsc{Department of Mathematics,}\par 
\textsc{University of Chicago,}\par
\textsc{5734 S. University Ave.}\par
\textsc{Chicago, IL 60637, U.S.A.} \par 
  \addvspace{\medskipamount} 
  \text{E-mail}: 
\texttt{chen@math.uchicago.edu} \par
}}

\begin{document}

\title{Analytic number theory for 0-cycles}

\author{Weiyan Chen}
\date{\today}

\begin{abstract}
There is a well-known analogy between integers and polynomials over $\fq$, and a vast literature on analytic number theory for polynomials.  From a geometric point of view, polynomials are equivalent to effective 0-cycles on the affine line. This leads one to ask: Can the analogy between integers and polynomials be extended to 0-cycles on more general varieties? In this paper we study prime factorization of effective 0-cycles on an arbitrary connected variety $V$ over $\fq$, emphasizing the analogy between integers and 0-cycles. For example, inspired by the works of Granville and Rhoades, we prove that the prime factors of 0-cycles are typically Poisson distributed. 
\end{abstract}

\maketitle

\renewcommand{\thefootnote}{\fnsymbol{footnote}} 
\renewcommand{\thefootnote}{\arabic{footnote}} 


\section{Introduction}


Riemann observed that the classical zeta function encodes  information about how a random integer factors into primes. The goal of this paper is to present concrete examples addressing the question: What does the zeta function of a variety $V$ tell us about how a random 0-cycle on $V$ factors into ``primes''? We study prime factorization of effective 0-cycles on a connected variety $V$ over $\fq$, emphasizing the analogy between integers and 0-cycles. The key input comes from a version of the Weil conjectures for varieties not necesarily complete or nonsingular, proved by Dwork, Grothendieck, Deligne \emph{et al}.



\subsection{Prime factorization of 0-cycles}
By a ``variety'' we mean an integral scheme of finite type. Let $q$ be a power of a prime number. Fix $V$ to be a geometrically connected variety over $\fq$. We do not require $V$ to be projective or nonsingular.  An \emph{effective 0-cycle} $C$ on $V$ over $\fq$ (we will just call it a ``0-cycle'' in this paper) is a formal $\NN$-linear sum
\begin{equation}
\label{factorization}
C=n_1P_1+n_2P_2+\cdots+n_lP_l \ \ \ \ \ \ \ \ \ \ n_i\in \NN
\end{equation}
of closed points $P_i$ on $V$. The \emph{degree} of $C$ is $\deg(C):=\sum_i n_i\deg(P_i)$ where $\deg(P_i)$ denotes the degree of the closed point. 
 We view equation (\ref{factorization}) as giving the \emph{prime factorization of 0-cycle} written additively, where closed points on $V$ play the role of ``primes''. The set of all 0-cycles on $V$ of degree $n$ is precisely $\sq$, the set of $\fq$-points on the $n$-th symmetric power $\sym^nV:=V^n/S_n$. The square-free 0-cycles \emph{i.e.} those with each $n_i=1$ in (\ref{factorization}) are  the $\fq$-points on the $n$-th configuration space $\conf^n V$ (see Section \ref{set up} below).


A well-known example is when $V$ is the affine line $\A^1$. There is a natural bijection between closed points on $\A^1$ over $\fq$ of degree $n$ and monic irreducible polynomials over $\fq$ of degree $n$. This bijection extends via equation (\ref{factorization}) to be between 0-cycles on $\A^1$ over $\fq$ and monic polynomials over $\fq$. Adding two 0-cycles on $\A^1$ corresponds to multiplying two polynomials.  In this way, 0-cycles on a variety $V$ generalize monic polynomials. 

We point out that $\sq$ is not the $n$-th symmetric power of $V(\fq)$. For example, the polynomial $x^2+1$ is an element in $\sym^2\A^1(\ff_{3})$,  though none of its roots lies in $\A^1(\ff_3)$.

The analogy between integers and polynomials over $\fq$ can be extended to 0-cycles on a connected variety $V$ over $\fq$. We summarize the correspondence in Table \ref{table} below.

\begin{table}[h!]
\caption{Integers { \emph{vs.}}\ \ 0-cycles  \ \ \ \ \ \ \ \ \ \ \ \ \ }\label{table}
\begin{tabular}{|l|l|}
\hline
Positive integers $x$          & 0-cycles $C$ on a variety $V$ over $\fq$                  \\ \hline
$\log x$                        &$\deg C$ \\ \hline
Multiplication of integers                        &Formal addition of 0-cycles                                        \\ \hline
Integers in $(N,2N]$ & $\sq$                                                      \\ \hline
Square-free integers in $(N,2N]$  & $\cvq$            \\ \hline
Prime numbers                     & Closed points on $V$
\\ \hline
Prime factorization of integers   & Prime factorization of 0-cycles \\ \hline
\end{tabular}
\end{table}

Note that the degree of 0-cycles corresponds to the logarithm of positive integers because $\deg(C+D)=\deg C+\deg D$, just as for integers, $\log(x\cdot y)=\log x+\log y$.  

\subsection{Summary of results}
There has been a vast literature on analytic number theory for polynomials over $\fq$. 
 Guided by the analogy in Table \ref{table}, we explore analytic number theory for 0-cycles: we study asymptotic statistics for the prime factorization of 0-cycles on a connected variety $V$ over $\fq$. Our results, to be summarized below, are  analogs of classical results in analytic number theory, and also generalize previous works about polynomials over $\fq$. Among these results, 
 Theorem \ref{intro Poisson} is the most difficult to prove. \\
 
 First, we start with the prime number theorem, proved by Hadamard and de la Vallée-Poussin, which says that the probability for a random integer in $(N,2N]$ to be prime is 
$$\sim\frac{1}{\log N} \ \ \ \ \ \ \text{as $N\to\infty$}.$$ 
We prove a similar statement for 0-cycles: there exists a constant $c$ depending on the variety $V$ such that the probability for a random 0-cycle in $\sq$ to be irreducible is 
\begin{equation}
\label{PNT direct}
\sim\frac{c}{\log |\sq|}\ \ \ \ \ \ \text{as $n\to\infty$}.
\end{equation}
More precisely, we prove the following statement with an explicit error bound:
\begin{prop}[\bf Prime number theorem for 0-cycles]
\label{PNT}
Suppose $V$ is a geometrically connected scheme of finite type over $\fq$ with $d:=\dim V\ge 1$. Let $\zt$ be the zeta function of $V$ and let $\zto:=\zt (1-q^{d}t)$. If
$$\mn:=\bigg|\{C\in\sq:\text{$C$ is prime}\}\bigg|,$$ 
then we have
$$\frac{\mn}{|\sq|}=\frac{1}{n}\frac{1}{\zod}+O(\frac{1}{nq^{n/2}}).$$
\end{prop}
The Weil conjectures guarantee $\zod$ to be a positive number. See Section \ref{set up} below. Proposition \ref{PNT} in the case $V=\A^1$ is a classical; it was first proved by Gauss. \\

Proposition \ref{PNT} gives the asymptotic probability for a 0-cycle to be prime. We further ask:
What is the probability for a  0-cycle to be ``nearly prime'' \emph{i.e.} to factor into a  product of large primes? and to be ``highly composite'', \emph{i.e.} to factor into a  product of small primes? The asymptotic answer to the two questions will be described by the following two functions $\RR_{\ge 1}\to\RR_{\ge0}$, respectively: the  \emph{Buchstab function} $\oo$  and the \emph{Dickman–de Bruijn function} $\rho$. See Section \ref{components} below for their definitions. The two functions are important because of the following classical theorems in analytic number theory:
\begin{itemize}
\item (Buchstab \cite{Buchstab})  If $\Phi(x,y)$ denotes the number of integers $\le x$ whose prime factors are all $\ge y$, then for any $u\ge1$, as $x\to\infty$
 $$\frac{\Phi(x,x^{1/u})}{x}\sim \frac{\oo(u)u}{\log x}.$$
 \item (Dickman \cite{Dickman}, with error term proved by  Ramaswami \cite{Ra}) If $\Psi(x,y)$ denotes the number of integers $\le x$ whose prime factors are all $\le y$, then for any $u\ge1$, as $x\to\infty$
$$\frac{\Psi(x,x^{1/u})}{x}=\rho(u)+O(1/\log x).$$
\end{itemize}
We prove the following analogs for 0-cycles:
\begin{prop}[\textbf{No small/large factor}]
\label{dickman}
Suppose $V$ is a geometrically connected scheme of finite type over $\fq$ with $d:=\dim V\ge 1$. Let $\Phi(x,m)$ denote the number of 0-cycles of degree $n$ on $V$ whose  prime factors all have degree $\ge m$, and $\Psi(n,m)$ denote the number of 0-cycles of degree $n$ on $V$ whose prime factors all have degree $\le m$. Then for every $ u>0$, as $n\to\infty$,
\begin{align}
\frac{\Phi(n,n/u)}{|\sq|} &\sim \frac{\oo(u)u}{n}\\
\frac{\Psi(n,n/u)}{|\sq|} &= \rho(u)+O(1/n)
\end{align}
\end{prop}

Proposition \ref{dickman} in the case when $V=\A^1$ was proved in \cite{BMPR} and \cite{OPRW}. \\

Next, we ask: Fix a prime $P$, what is the probability for a 0-cycle to be divisible by $P$? More generally, how does the order of $P$ \emph{i.e.} the coefficients in (\ref{factorization}) in a random 0-cycle distribute? We first consider the same questions for integers. For a prime number $p$, the $p$-adic order $\nu_p$ (\emph{i.e.} the order of $p$ in the prime factorization) can be viewed as a random variable on $(N,2N]$. A straightforward calculation gives the asymptotic distribution of $\nu_p$ on (square-free) integers in $(N,2N]$ as $N\to\infty$:
\begin{enumerate}
\item $\pr(\nu_p(m)=j:m\in (N,2N])\longrightarrow p^{-j}(1-p^{-1}),$ geometric distribution.
\item $\pr(\nu_p(m)=1:\text{square-free }m\in (N,2N])\longrightarrow (p+1)^{-1}$, Bernoulli distribution.
\item\label{3} If $p$ and $l$ are different primes, then $\nu_p$ and $\nu_l$ on $(N,2N]$, or on square-free integers in $(N,2N]$, tend to be independent. 
\end{enumerate}
We prove the corresponding statements for 0-cycles.
\begin{prop}[\bf Asymptotic distribution of prime orders in 0-cycles]
\label{rv}
Suppose $V$ is a geometrically connected scheme of finite type over $\fq$ with $d:=\dim V\ge 1$.  There exists a constant $b$ depending only on $V$ satisfying the following. For a closed point $P$ on $V$ and a 0-cycle $C$,  define $\nu_P(C)$ to be the order of $P$  in the prime factorization of $C$ as in (\ref{factorization}). Let $k:=\deg(P)$. 
\begin{enumerate}
\item[$(1^\prime)$] The random variable $\nu_P$ on $\sq$ converges to the geometric distribution as $n\to\infty$. Precisely, for any natural number $j$, we have the following estimate for probability:
$$\pr(\nu_P(C)=j:C\in\sq)=\bigg(\frac{1}{q^{dk}}\bigg)^j\bigg(1-\frac{1}{q^{dk}}\bigg)\bigg(1+O(\frac{n^b}{q^n})\bigg).$$
\item[$(2^\prime)$] The random variable $\nu_P$ restricted to $\cvq$ converges to the Bernoulli distribution as $n\to\infty$. Precisely, the probability
$$\pr(\nu_P(C)=1:C\in\cvq)=\frac{1}{q^{dk}+1}\bigg(1+O(\frac{n^b}{q^n})\bigg).$$
\item[$(3^\prime)$] If $P$ and $Q$ are different closed points, then $\nu_P$ and $\nu_Q$ on $\sq$ or on $\cvq$ tend to be independent random variables as $n\to\infty$. Precisely, 
\begin{align*}
\pr (\nu_P= i \text{ and } \nu_Q= j)=\pr(\nu_P= i)\cdot\pr(\nu_Q= j)\cdot \bigg(1+O(\frac{n^b}{q^n})\bigg),
\end{align*}
where the probability is taken over either $\sq$ or $\cq$.
\end{enumerate}
\end{prop}

Proposition \ref{rv} in the case $V=\A^1$ was proved by Arratia-Barbour-Tavare (Theorem 3.1 in \cite{ABT}). Proposition \ref{rv} in the case when $V$ is the affine or the projective space can also be deduced from a theorem of Poonen (Theorem 1.1 in \cite{Po}). 
%
Recently, Farb-Wolfson and the author independently generalized a theorem of Church-Ellenberg-Farb \cite{CEF} about asymptotic arithmetic statistics on $\conf^n\A^1$ to that on $\cv$ for $V$ a smooth variety (see Theorem B in \cite{FW} and Corollary 4 in \cite{me}). Proposition \ref{rv} gives a probabilistic interpretation and a new proof of this generalization, and removes the assumption for $V$ to be smooth.
See Section \ref{stats} below for more details.

\begin{rmk}[\bf Erd\H{o}s-Kac's heuristic]
\label{heu}
Erd\H{o}s and Kac made the following brilliant observation. Let $\OO(m)$ denote the total number of prime factors of an integer $m\in(N,2N]$, counted with multiplicities. Then $\OO(m)=\sum_p \nu_p(m)$, summing over all prime numbers $p$. Since $\nu_p$ tends to independent random variables as $N\to\infty$ by the statement (\ref{3}) above,  heuristically $\OO$ is a sum of independent random variable  in the limit. By the Central Limit Theorem, one should expect  $\OO$ to tend to the normal distribution as $N\to\infty$. This observation leads Erd\H{o}s and Kac to prove their celebrated theorem in \cite{EK} 
  which roughly says that when $N$ is large, $\OO$ is closed to be normally distributed with mean and variance  $\sim\log(\log N)$.\footnote{Erd\H{o}s-Kac theorem was originally about $\omega$, the number of distinct prime factors. But the same heuristic applies and the same result holds for $\OO$.}
  
  By Proposition \ref{rv}, the exact same heuristic applies to 0-cycles. For $C$ a 0-cycle, define $\OO(C)$ to be the number of factors in the prime factorization of $C$ counted with multiplicities.  Then
$$\OO(C)=\sum_{P\in V^{cl}}\nu_P(C),$$
where $V^{cl}$ is the set of closed points of $V$. By Porposition \ref{rv}, the sequence $(\nu_P)_{P\in V^{cl}}$ converges to a sequence of independent random variables as $n\to\infty$. As before, one should expect that when $n\to\infty$, the distribution of $\OO$ on $\sq$ would be close to normal with mean and variance $\sim\log n$ (recall that $n$ corresponds to $\log N$ in Table \ref{table}). This heuristic is confirmed by a theorem of  Liu (see Corollary 2 in \cite{Liu}).\\ 
\end{rmk}

Erd\H{o}s-Kac's heuristic and Liu's theorem imply that a 0-cycle $C$ is expected to contain $\sim\log \deg C$ many prime factors. Thus, for $C$ with prime factorization $C=n_1 P_1+\cdots n_lP_l$,
$$\phi(C):=\Big\{\log\deg P_1,\ \log \deg P_2,\ \cdots,\ \log\deg P_l\Big\}$$
is typically a collection of $\sim\log \deg C$ many numbers in the interval $[0,\log\deg C]$. How should we expect $\phi(C)$ to distribute in $[0,\log\deg C]$?

We first consider the same question for integers. Granville (Theorem 1 in \cite{G1}) proved that for almost all integers $x$, the sets of  numbers $$\phi(x):=\{\log(\log p): p|x \}$$
are close to being ``random'' \emph{i.e.} Poisson distributed. A sequence of finite sets $S_1,S_2,\cdots$ is said to be \emph{Poisson distributed} (see \cite{G1}) if there exist functions $m_j,K_j,L_j\to\infty$  monotonically as $j\to\infty$ such that $S_j\subset [0,m_j]$ and $|S_j|\sim m_j$; and for all $L$ with $L\in[1/L_j,L_j]$ and all integer $k$ with $k\in [0,K_j]$, we have 
$$\pr\bigg(t\in [0,m_j] : \Big|S_j\cap [t,t+L]\Big|=k\bigg) \sim e^{-L}\frac{L^k}{k!},$$
where $\pr$ stands for probability with respect to the Lebesgue measure. 
For example, if $S_j$ is a set of $j$ real numbers chosen uniformly and independently in the interval $[0,j]$, then the sequence $S_j$ is almost surely Poisson distributed. In this paper, we prove the following analog of Granville's theorem for 0-cycles.

%
%
%

\begin{thm}[\textbf{Prime factors in 0-cycles are Poisson distributed}]
\label{intro Poisson}
Suppose $V$ is a geometrically connected scheme of finite type over $\fq$ with $d:=\dim V\ge 1$.  As $n\to\infty$, for almost every 0-cycle $C$ in $\sq$ with prime factorization $C=n_1 P_1+\cdots n_lP_l$, the set of numbers 
\begin{equation}
\label{phi}
\phi(C):=\Big\{\log\deg P_1,\ \log \deg P_2,\ \cdots,\ \log\deg P_l\Big\}
\end{equation}
is approximately Poisson distributed. More precisely, there exist functions $K(n), L(n)\to\infty$ monotonically as $n\to\infty$ such that for all $\ep>0$, and all $n$ sufficiently large depending on $\ep$, with $\pr$ representing  probability under the Lebesgue measure, we have
$$(1-\ep)e^{-L}\frac{L^k}{k!}\le  \pr\bigg(t\in[0,\log n] : \Big|\phi(C)\cap [t,t+L]\Big|=k\bigg)\le(1+\ep)e^{-L}\frac{L^k}{k!}$$
for all $L\in[1/L(n),L(n)]$ and all integer $k\le K(n)$, and for at least $(1-\ep)|\sq|$ many 0-cycles $C$ in $\sq$. 
\end{thm}
\begin{rmk}[\textbf{Related work}]
Theorem \ref{intro Poisson} in the case when $V=\A^1$ was first proved by Rhoades  (Theorem 1.3 in \cite{Rh}).  Following the theme that integers and permutations should have similar statistical behaviors, Granville proved that the cycle lengths of permutations are also typically Poisson (Theorem 1 in \cite{G3}). See Granville's excellent survey \cite{G} for more on the anatomy of integers and permutations. Our proof of Theorem \ref{intro Poisson} uses ideas from both Rhoades' and Granville's works. 
\end{rmk}

\begin{rmk}[\textbf{Proof methods}]
The  key ingredient in all the results above comes from the Weil conjectures, which are now theorems thanks to the works of Dwork, Grothendieck, Deligne \emph{et al}. In addition,  to prove Proposition \ref{dickman}, we use general results about decomposable combinatorial structures proved in \cite{BMPR} and \cite{OPRW}. To prove Theorem \ref{intro Poisson}, we establish a comparison lemma (Lemma \ref{comparison} below) relating statistics about permutations and  about 0-cycles, and then apply Granville's theorem in \cite{G3} about cycle lengths in permutations. 
\end{rmk}

\section*{Acknowledgment}
The author would like to thank Andrew Barbour, Jordan Ellenberg, Matthew Emerton, Andrew Granville, Sean Howe, Dan Petersen, and Jesse Wolfson for helpful conversations, and to  thank Sean Howe, Jeffrey Lagarias, Andrew Sutherland and an anonymous referee for suggestions on an earlier draft. The author is deeply grateful to his advisor Benson Farb, both for his kind  support of this project and for his detailed comments on an earlier draft of this paper.

\section{Symmetric powers, Zeta function, and the Weil conjectures}
\label{set up}

In this section we recall the Weil conjectures. All result presented in this section is  known. 

The \emph{$n$-th symmetric power} of a variety $V$ is the quotient
$$\sv:=V^n/S_n,$$
where the symmetric group $S_n$ acts on $V^n$ by permuting the coordinates. $\sv$ is also a variety over $\fq$ (see \cite{Mumford}, page 66). 
An $\fq$-point on $\sv$ is precisely a 0-cycle of degree $n$ on $V$ defined over $\fq$. 
Similarly, the \emph{$n$-th configuration space} is defined to be
$$\cv:=\{(x_1,\cdots,x_n)\in V^n: x_i\ne x_j,\ \forall i\ne j\}/S_n,$$
where $S_n$ also acts on permuting the coordinates. 
$\cv$ is a subvariety of $\sv$. An $\fq$-point on $\cv$ is precisely a square-free 0-cycle of degree $n$ on $V$ over $\fq$. We first recall a famous theorem of Lang-Weil.
\begin{thm}[Lang-Weil, Theorem 1 in \cite{LW}]
\label{LW}
Suppose $V$ is a subvariety of a  projective space defined over $\fq$  with $d:=\dim V$. We have 
\begin{equation}
\label{es V} \frac{|V(\ff_{q^n})|}{q^{nd}}=1+O(\frac{1}{q^{n/2}}).
\end{equation}
\end{thm}




Lang-Weil assumed $V$ to be a subvariety of a projective space. The following theorem will imply that (\ref{es V}) also holds for any geometrically connected schemes of finite type.

\begin{thm}[Dwork, Grothendieck, Deligne \emph{et al}]
\label{wc}
Suppose $V$ is a geometrically connected scheme of finite type over $\fq$ with $d:=\dim V$. Let $\zt$ be the \emph{zeta function} of $V$:
\begin{equation}
\label{def zeta}
\zt:=\exp\bigg({\sum_{k=1}^\infty \frac{|V(\ff_{q^k})|}{k}t^k\bigg)}= \sum_{n=0}|\sq|t^n.
\end{equation}
\begin{enumerate}[label=(\Roman*)]
\item
\label{ra} There exist polynomials $P_i(t)$ for $i=0,\cdots, 2d$ such that 
$$\zt=\frac{P_1(t)\cdots P_{2d-1}(t)}{P_0(t)\cdots P_{2d}(t)},$$
where $P_{2d}(t)=1-q^dt$.
\item
\label{rh}  For each $i$, for each $\alpha$ such that $P_i(\alpha)=0$, there exists some $j\leq i$ such that $|\alpha|=q^{-j/2}$. It is possible that $j$ depends on $\alpha$.
\end{enumerate}
\end{thm}
When $V$ is nonsingular and projective, the equality $j=i$ will be achieved for all $\alpha$ in \ref{rh}; in this case, \ref{ra} and \ref{rh} are parts of the \emph{Weil conjectures}, orginially formulated by Weil \cite{Weil}. Part \ref{rh},  due to Deligne (see Th\'eor\`eme I (3.3.1) in \cite{D}), are often called the ``Riemann Hypothesis over finite fields''; it will be especially important for our purposes.

\section{Prime number theorem and asymptotic distribution of prime orders}
In this section we will prove Propositions   \ref{PNT} and \ref{rv} stated in the introduction.

\subsection{Proof of Proposition \ref{PNT}}

We will prove Lemma \ref{es all} and \ref{es mn}, which give estimates for $|\sq|$ and $\mn$, respectively. Taking the quotient, we obtain Proposition \ref{PNT}.  

\begin{lemma}\label{es all}
Suppose $V$ is a connected variety over $\fq$ of dimension $d$.  Define $\zto:=\zt (1-q^{d}t)$. There exists a constant $b$ depending only on $V$ such that
\begin{align}
\label{es sym} & \frac{|\sq|}{q^{nd}}=\zod +O( \frac{n^b}{q^{n}}).
\end{align}
\end{lemma}
\begin{proof}[Proof of Lemma \ref{es sym}]
Theorem \ref{wc} gives that  $\zto:=\zt(1-q^dt)$ is a rational function in $t$ with radius of convergence at least $q^{-\frac{2d-2}{2}}=q^{-(d-1)}$. In particular, $\zto$ converges at $t=q^{-d}$. We decompose  $\zt$ into partial fractions of the following form:
\begin{equation}
\label{partial}
\zt = \underbrace{\frac{\zod}{1-q^dt}}_\text{dominating term}+\underbrace{\frac{\zto-\zod}{1-q^dt}}_\text{remainder}.
\end{equation}
The dominating term in (\ref{partial}), which is a geometric series, tells that $|\sq|$ grows like $\zod q^{nd}$, contributing to the first summand in (\ref{es sym}) 
The remainder in (\ref{partial}) is a rational function in $t$ with radius of convergence at least $q^{-\frac{2d-2}{2}}=q^{-(d-1)}$. Thus, its $n$-th Taylor coefficient grows like 
$O( {n^b}{q^{n(d-1)}})$ where $b$ is the number of poles with norm $q^{-(d-1)}$ (see Theorem IV.9 in \cite{AC}), which after normalization contributes to  the error term in (\ref{es sym}).
\end{proof}

\begin{rmk}[\textbf{The constant $b$}]
From the proof, the constant $b$ is equal to the number of poles of $\zt$ with norm $q^{-(d-1)}$, which further depends on the compactly supported \'etale cohomology of $V$ in dimension $2d-2$. 
\end{rmk}

\begin{lemma}\label{es mn}
Suppose $V$ is a connected variety over $\fq$ of dimension $d$, then
$$
\frac{\mn}{q^{nd}}=\frac{1}{n} +O(\frac{1}{nq^{\frac{n}{2}}}).
$$
\end{lemma}
\begin{proof}[Proof of Lemma \ref{es mn}]
Since the $\ff_{q^n}$-points on $V$ are precisely the fixed points  of $\mathrm{Frob}_{q^n}$, we have:
$$|V(\ff_{q^n})| = \sum_{k|n}k \mk.$$
Applying M\"obius inversion, we obtain
\begin{equation}
\label{mobius}
\mn = \frac{1}{n}\sum_{k|n}\mu({n/k})|V(\ff_{q^k})|.
\end{equation}
This formula together with (\ref{es V}) establishes the lemma.
%
\end{proof}

\subsection{Proof of Proposition \ref{rv}}
 Let $P$ be a closed point of $V$ with degree $k$. \\

\noindent \textbf{Step (1): $\pmb{\nu_P}$ on $\pmb{\sq}$.} For any $C\in \sq$ and any natural number $j$, 
$$\nu_P(C)\ge j \iff C=\underbrace{x+\ldots+x}_{j\text{-times}}+C^\prime\ \ \ \text{for some $C^\prime \in \mathrm{Sym}^{n-jk}V(\fq)$}.$$
Therefore, we have 
\begin{align*}
\pr(\nu_P\geq j:\ \sq)&=\frac{|\mathrm{Sym}^{n-jk}V(\fq)|}{|\sq|}\\
&= \frac{1}{q^{jkd}}\cdot \bigg(1+O(\frac{n^b}{q^n})\bigg)&\text{by (\ref{es sym})}.
\end{align*}
The claim is established by taking $$\pr\Big(\nu_P= j:\ \sq\Big)=\pr\Big(\nu_P\geq j:\ \sq\Big)-\pr\Big(\nu_P\geq j-1:\ \sq\Big).$$\\

\noindent\textbf{Step (2): $\pmb{\nu_P}$ on $\pmb{\cq}$.} We first prove a lemma on the size of $|\cq|$.

\begin{lemma} \label{sf}
There exists a constant $b$ (which is the same as in Lemma \ref{es all}) such that 
\begin{align}
\label{sf1}
\frac{|\cq|}{q^{nd}}&=\frac{\zod}{Z(V,q^{-2d})} +O(\frac{n^b}{q^n}).
\end{align}
\end{lemma}
\begin{proof}[Proof of Lemma \ref{sf}]
\begin{align*}
\sum_{n=0}^\infty |\cq|t^n&=\prod_{x\in V^{cl}}(1+t^{\deg x})=\prod_{x\in V^{cl}}\frac{1-t^{2\deg x}}{1-t^{\deg x}}=\frac{\zt}{Z(V,t^2)}.
\end{align*}
Now$\zt/Z(V,t^2)$ is the generating function for $|\cq|$. By part \ref{rh} of Theorem \ref{wc}, $\zt/Z(V,t^2)$ is rational  with $t=q^{-d}$ the unique pole with the smallest absolute value. The lemma follows by the same partial fraction method as in the proof of (\ref{es sym}).
\end{proof}
In particular, by Lemma \ref{sf}, there is a constant $m$ such that for any integers $n$ and $j$,
\begin{equation}
\label{sf2}
\bigg|\frac{|\conf^{n-j}V(\fq)|}{|\cq|}-\frac{1}{q^{jd}}\bigg|\le \frac{m}{q^{jd}}(\frac{n^b}{q^n}).
\end{equation}

Define $$\conf^n_xV(\fq):=\{C\in \cq : \nu_P(C)=1\}.$$
Observe that for any $C\in \cq$,
$$\nu_P(C)=1 \iff C=x+C^\prime\ \ \ \text{for some $C^\prime \in \conf^{n-jk}V(\fq)\setminus \conf^{n-jk}_xV(\fq)$}.$$
Counting elements give:
\begin{equation}
\label{inc}
|\conf^n_xV(\fq)|= |\conf^{n-jk}V(\fq)|- |\conf^{n-jk}_xV(\fq)|.
\end{equation}
Therefore, we have:
\begin{align*}
&\pr(\nu_P=1:\  \cq)=\frac{|\conf^n_xV(\fq)|}{|\cq|}\\
&= \frac{|\conf^{n-jk}V(\fq)|-|\conf^{n-2jk}V(\fq)|+|\conf^{n-3jk}V(\fq)|-\cdots}{|\cq|}&\text{by (\ref{inc}) iteratively}\\
&= ({q^{-jkd}-q^{-2jkd}+q^{-3jkd}-\cdots})\cdot\bigg(1+O(\frac{n^b}{q^n})\bigg)&\text{by (\ref{sf2})}\\
&=\frac{1}{1+q^{jkd}}\bigg(1+O(\frac{n^b}{q^n})\bigg).
\end{align*}\\

\noindent\textbf{Step (3): Independence.}
For any distinct closed points $P$ and $Q$ on $V$, the same calculation as above gives that for random $C$ in $\sq$ or in $\cq$,
\begin{align*}
\pr \Big(\nu_P\ge i \text{ and } \nu_Q\ge j\Big)=\pr\Big(\nu_P\ge i\Big)\cdot\pr\Big(\nu_Q\ge j\Big)\cdot \bigg(1+O(\frac{n^b}{q^n})\bigg).
\end{align*}
\qed

\subsection{Statistics weighted by character polynomials}\label{stats}

In this section we discuss a corollary of Proposition \ref{rv}. For each positive integer $k$, define $X_k:\coprod_{n=1}^\infty S_n\to\ZZ$ where $X_k(\s)$ is the number of $k$-cycles in the unique cycle decomposition of $\s\in S_n$.  A \emph{character polynomial} is a polynomial $P\in\QQ[X_1,X_2,\cdots]$. It defines a class function on $S_n$ for all $n$.  
 For every 0-cycle $C$ of degree $n$, choose a permutation $\s_C\in S_n$ unique up to conjugacy such that
\begin{equation}\label{s_C}
\forall k=1,2,\cdots n, \ \ \ \ \ \ \ \  X_k(\s_C) = \text{number of prime factors of degree $k$ in $C$}.
\end{equation}
The map $C\mapsto\s_C$ gives the following function:
\begin{equation}\label{pushforward}
\sq\to \{\text{conjugacy classes in $S_n$}\}.
\end{equation}
In particular, if $C\in\cq$, then $\s_C$ records how the Frobenius automorphism permutes the $n$ distinct points in $C$. Every character polynomial $P$ can be viewed as a random variable on $\sq$ or on $\cq$ for all $n$, via the map $C\mapsto P(\s_C)$. Church-Ellenberg-Farb used the theory of representation stability to prove that for any character polynomial $P$, the limit of the expected value of $P$ over $\conf^n\A^1(\fq)$:
$$\lim_{n\to\infty}\E[P;\conf^n\A^1(\fq)]$$
always exists [\cite{CEF}, Theorem 1]. This result was later generalized by Farb-Wolfson [ \cite{FW}, Theorem C] and by Chen [\cite{me}, Corollary 4] from $\conf^n\A^1$ to $\conf^nV$ for any smooth, connected variety $V$. The approaches in \cite{CEF} and in \cite{FW} were algebro-geometric, using the stability of \'etale cohomology of $\cv$ with twisted coefficients. The approach in \cite{me} was analytic: the limit of $\E[P;\cq]$ was calculated from the zeta function of $V$. In the following, we will give a new proof of the generalized result from a probabilistic point of view. Furthermore, the original assumption for $V$ to be smooth (as was needed in \cite{FW} and \cite{me}) is removed in this new proof.



\begin{cor}\label{joint moment}
Suppose $V$ is a geometrically connected scheme of finite type over $\fq$ with $d:=\dim V\ge 1$.   Let $\ld=(\ld_1,\cdots,\ld_l)$ be any sequence of nonnegative integers, and let
$${X\choose\ld}:=\prod_{k=1}^l{X_k\choose \ld_k},$$
considered as a random variable on $\sq$ and on $\cq$. Then the following expected values have explicit asymptotic formulas:
\begin{align}
\label{sym rv}\lim_{n\to\infty}\E\bigg[{X\choose \ld};\sq\bigg]&=\prod_{k=1}^l{\mk+\ld_k-1\choose \ld_k}\frac{1}{(q^{kd}-1)^{\ld_k}}\\
\label{conf rv}\lim_{n\to\infty}\E\bigg[{X\choose \ld};\cq\bigg]&=\prod_{k=1}^l{\mk\choose \ld_k}\frac{1}{(q^{kd}+1)^{\ld_k}}.
\end{align}
\end{cor}
Every character polynomial $P$ can be written as a linear combination of polynomials of the form $X\choose \ld$. Thus Corollary \ref{joint moment} implies that $\lim_{n\to\infty}\E[P;\cq]$  converges for any character polynomial $P$.

\begin{proof}
Observe that for any $C\in\sq$, we have
$$
X_k(\s_C)= \text{total number of factors in $C$ of degree $k$}=\sum_{P:\deg P=k}\nu_P(C),
$$
where the sum is over all closed points $P$ of degree $k$.  By Proposition \ref{rv},  the random variable $X_k$ on $\sq$ converges to a sum of $\mk$ many independent identically distributed (i.i.d.) geometric random variables. Therefore $X_k$ converges to the negative binomial distribution. Similarly, the random variable $X_k$ on $\cq$ converges to a sum of $\mk$ many i.i.d. Bernoulli random variables, and therefore tends to the binomial distribution. The well-known formulas for the factorial moments of the limiting negative binomial and binomial distributions give the right hand sides of  (\ref{sym rv}) and  (\ref{conf rv}), respectively.
\end{proof}


\section{0-cycles as decomposable combinatorial structures}\label{log}
In this section we prove Proposition \ref{dickman} stated in the introduction. We apply the powerful axiomatic results from the theory of decomposable combinatorical structures in  \cite{BMPR}, and \cite{OPRW}. The axioms because of the Weil conjectures.

\subsection{Abstract decomposable combinatorial structures}
\label{dcs}
Consider the following general set-up. Let $\CCC$ be a set of  combinatorial objects where each element has a nonnegative degree. Define $\PPP$ to be the set of all multisets of elements in $\CCC$, where the degree of each multiset is defined to be the sum of the degree of its elements with multiplicities. Then the triple  $(\CCC,\PPP, \text{degree})$ is called a \emph{decomposable combinatorial structure} (see introduction in \cite{FS}). Objects in $\PPP$ can be viewed as ``composites'', while objects in $\CCC$ can be viewed as ``primes''. 
Given a decomposable combinatorial structure $(\CCC,\PPP, \text{degree})$, define
\begin{align}
\nonumber\CCC_n&:=\{\text{elements in $\CCC$ of degree $n$}\},&\PPP_n&:=\{\text{elements in $\PPP$ of degree $n$}\},
\\
\nonumber C_n&:=|\CCC_n|,& P_n&:=|\PPP_n|,\\
\label{dcs definition}C(t)&:=\sum_{n=0}^\infty C_nt^n,& P(t)&:=\sum_{n=0}^\infty P_nt^n. 
\end{align}
We also consider $\DDD\subseteq \PPP$ which consists of sets (instead of multisets) of distinct elements in $\CCC$.   We define $\DDD_n$ and $D_n$ similarly. See \cite{FS} for many examples of decomposable combinatorial structures that arise naturally in different areas of mathematics. In this paper, we will focus on the following example.

\begin{ex}[\bf 0-cycles as decomposable combinatorial structures]
\label{dcsex}
For each $n$, let $\PPP_n:=\sq$ and $\CCC_n:=\{\text{closed points on $V$ of degree $n$}\}$. Then $(\CCC,\PPP,\deg)$ is a decomposable combinatorial structure. In this case, $P(t)$ is precisely the zeta function $\zt$ of the variety $V$. Moreover, $\DDD_n=\cq$. 
\end{ex}
We now apply general results of decomposable combinatorial structures to study 0-cycles.

\subsection{Proof of Proposition \ref{dickman}}
\label{components}
First, we state the definitions of the Buchstab and the Dickman–de Bruijn functions. 
\begin{defn} The \emph{Buchstab function} $\oo:\RR_{\ge 1}\to\RR_{\ge0}$ is the unique continuous function satisfying the differential equations:
\begin{align*}
&\oo(u)=1/u,&1\le u\le2,\\
&\oo(u)+u\oo^\prime(u)=\oo(u-1), & u\ge2.
\end{align*}
The \emph{Dickman–de Bruijn function} $\rho:\RR_{\ge 0}\to\RR_{\ge0}$ is the unique continuous function satisfying the differential equations:
\begin{align*}
&\rho(u)=1&0\le u\le1,\\
&u\rho^\prime(u)+\rho(u-1)=0&u\ge2.
\end{align*}

\end{defn}
Proposition \ref{dickman} will be a consequence of Lemma \ref{es all} together with the following general theorems about decomposable combinatorial structures. Given a decomposable combinatorial structure $(\CCC,\PPP, \deg)$ as in the previous section, define 
\begin{align*}
&\Phi(n,m):=|\{p\in P_n : p \text{ contains only primes of degree $\ge m$}\}|,\\
&\Psi({n,m}):=|\{p\in P_n : p \text{ contains only primes of degree $\le m$}\}|.
\end{align*}

\begin{thm}[Bender-Mashatan-Panario-Richmond, Omar-Panario-Richmond-Whitely]\label{general}
Assume there exist constants $K$ and $R$ such that as $n\to\infty$,
\begin{equation}
\label{dickman condition}
C_n \sim \frac{1}{n}R^n,\ \ \ \ \ \ \ \ \  \ P_n \sim KR^n.
\end{equation}
\begin{enumerate}
\item \label{buchstab}(Theorem 1.1 in \cite{BMPR}) For any $\ep>0$, and any $m,n$ with $\ep\le m/n\le 1$,
$$\frac{C_n}{\Phi(n,m)}\sim\frac{m}{n\oo(n/m)}$$
\item (Theorem 1 in \cite{OPRW}) For any $\ep>0$ and any $m,n$ with $\ep\le m/n\le 1$,
$$\frac{\Psi({n,m})}{P_n}=\rho(n/m)+O(1/m).$$
\end{enumerate}
\end{thm}

In our case (Example \ref{dcsex}), we have $C_n=\mn$, $P_n=|\sq|$. By Lemmas \ref{es all} and  \ref{es mn},  conditions (\ref{dickman condition}) are satisfied. Thus Theorem \ref{general} implies that for any $u\ge1$, as $n\to\infty$ we have
\begin{align*}
\frac{\Phi(n,n/u)}{|\sq|} &= \frac{\Phi(n,n/u)}{\mn}\frac{\mn}{|\sq|}\sim\frac{u\oo(u)}{n}\\
\frac{\Psi(n,n/u)}{|\sq|} &= \rho(u)+O(1/n).
\end{align*}
\qed

\section{Prime factors of 0-cycles are Poisson distributed}\label{section Poisson}
In this section we prove Theorem \ref{intro Poisson} by proving the following quantitative statement:

\begin{thm}\label{Poisson}
Suppose $V$ is a geometrically connected scheme of finite type over $\fq$ with $d:=\dim V\ge 1$.  For every $n$ large enough, let $y:=\log\log n/(\log \log \log n)^2$. There exists a small  subset $\Sigma \subset \sq$ with $|\Sigma|\leq|\sq|2^{-y/7}$ such that for every $C\in \sq\setminus \Sigma$, for every $L\in [1/y,y/20]$, 
and for every $r\le y (\log y)^{-2}$,
\begin{equation}\label{bound}
\frac{1}{\log n}\ld\bigg(\Big\{t\in[0,\log n] : \Big|\phi(C)\cap [t,t+L]\Big|=k\Big\}\bigg)=e^{-L}\frac{L^r}{r!}\bigg[1+O\Big(\frac{1}{2^{y/15}}\Big)\bigg] 
\end{equation}
where $\ld$ denotes the Lebesgue measure and $\phi(C)$ is as in (\ref{phi}).
\end{thm}

\begin{rmk}[\textbf{Notations}]\label{notation}
Before proving Theorem \ref{Poisson}, we first remark on notations. For any permutation $\s\in S_n$, let $\mu_\s$ denote the partition of $n$ given by the cycle type of $\s$. For any 0-cycle $C\in \sq$, abbreviate $\mu_C:=\mu_{\s_C}$ where $\s_C$ was defined in  (\ref{s_C}).

Suppose $\chi$ is an arbitrary class function  of $S_n$. For any $C\in\sq$, define $\chi(C):=\chi(\s_C)$. For any partition $\mu\vdash n$, define  $\chi(\mu):=\chi(\s)$ for any $\s\in S_n$ with cycle type $\mu$.  In this way, $\chi$ can be viewed as a random variable on three different sample space: on $S_n$, on $\sq$, and on $\cq$. We will write $\E[\chi, S]$ for the expected value of $\chi$ on the space $S$. 
\end{rmk}

\begin{proof}[Proof of Theorem \ref{Poisson}]
We prove the theorem in three steps: first, we prove a general lemma comparing statistics about  0-cycles and about permutations; second, we apply the comparison lemma to a theorem of Granville about permutations, and obtain an estimate for the left hand side of (\ref{bound}); finally, we construct the set $\Sigma\subset \sq$, following an argument due to Rhoades \cite{Rh}. The key input in the whole proof is to establish Lemma \ref{comparison} in the first step. \\

\noindent \textbf{Step 1: Compare permutations and square-free 0-cycles.}
We first prove a general comparison lemma in Lemma \ref{comparison}, which roughly says that  statistics on square-free 0-cycles are bounded by the corresponding statistics on permutations. 

\begin{lemma}\label{bound Sn}
Suppose $V$ is a geometrically connected scheme of finite type over $\fq$ with $d:=\dim V\ge 1$.  For each positive integer $n$, define a class function $g_n$ of $S_n$ as follows: for any partition $\mu=1^{\mu_1}2^{\mu_2}\cdots n^{\mu_n}$ of $n$, 
$$g_n(\mu):=\prod_{j=1}^n {\mj\choose \mu_j}\frac{j^{\mu_j}\mu_j!}{q^{dj\mu_j}}.$$
Then there exists a constant $B$ depending only on $V$ such that $\E[g_n^2, S_n]\leq B$ for all $n$.
\end{lemma}
\begin{proof}
\begin{align*}
\E[g_n^2, S_n] &=\frac{1}{n!}\sum_{\s\in S_n}\bigg(\prod_{j=1}^n {\mj\choose X_j(\s)}\frac{j^{X_j(\s)}X_j(\s)!}{q^{djX_j(\s)}}\bigg)^2\\
&\leq \frac{1}{n!}\sum_{\s\in S_n}\prod_{j=1}^n(\frac{j\mj}{q^{dj}})^{2X_j(\s)}\\
&\leq \frac{1}{n!}\sum_{\s\in S_n}\prod_{j=1}^n(1+A q^{-j/2})^{X_j(\s)}&\text{for some constant $A$ depending only on $V$}
\end{align*}
where the last inequality comes from Lemma \ref{es mn}. We will prove the following claim.

\begin{claim}
For each $n$, we have
$$\frac{1}{n!}\sum_{\s\in S_n}\prod_{j=1}^n(1+A q^{-j/2})^{X_j(\s)} = \sum_{i=0}^n\mul{A}{i}q^{-i/2}.$$
Recall that $\mul{m}{n}:={{m+n-1}\choose n}$ is the number of ways to choose $n$ elements from a set of size $m$, allowing repetition.
\end{claim}

\begin{proof}[Proof of the claim]
We will prove by induction on $n$. The case when $n=1$ is easily checked. For induction, we will use the following fact from combinatorics.
\begin{prop}[folklore]\label{cycle index}
The function 
$$Z_n:=\frac{1}{n!}\sum_{\s\in S_n}\prod_{j=1}^n a_j^{X_j(\s)},$$
 satisfies the following recurrence relation:
$Z_n=\frac{1}{n}\sum_{l=1}^n a_lZ_{n-l}.$
The polynomial $Z_n$ in the formal variables $a_j$'s is often known as the \emph{cycle index} of $S_n$.
\end{prop}

We apply the proposition by substituting the formal variable $a_j$ by $1+A q^{-j/2}$.
\begin{align*}
Z_n&:=\frac{1}{n!}\sum_{\s\in S_n}\prod_{j=1}^n(1+A q^{-j/2})^{X_j(\s)}\\
&= \frac{1}{n}\sum_{l=1}^{n}(1+A q^{-l/2})Z_{n-l}\ \ \  \ \ \ \ \ \ \ \ \ \ \ \ \ \ \ \ \ \ \ \ \ \ \text{by the proposition}\\
&=\frac{1}{n}\sum_{l=1}^{n}(1+A q^{-l/2})\sum_{i=0}^{n-l}\mul{A}{i}q^{-i/2}\\
&=\frac{1}{n}\bigg[\sum_{l=1}^{n}\sum_{i=0}^{n-l}\mul{A}{i}q^{-i/2} +A \sum_{l=1}^{n}\sum_{i=0}^{n-l}\mul{A}{i}q^{-(i+l)/2}\bigg]\\
&=\frac{1}{n}\Bigg[\sum_{k=0}^{n-1}(n-k)\mul{A}{k}q^{-k/2} + \sum_{k=1}^{n}\bigg[\sum_{l=1}^{k}\mul{A}{k-l}A\bigg]q^{-k/2}\Bigg]\\
&=\frac{1}{n}\sum_{k=0}^{n}\bigg[(n-k)\mul{A}{k} + \sum_{l=1}^{k}\mul{A}{k-l}A \bigg]q^{-k/2}\\
&=\sum_{k=0}^{n} \mul{A}{k}q^{-k/2}\ \ \ \ \ \ \ \ \ \ \ \ \ \ \ \ \ \ \ \ \ \ \ \ \ \text{since $\sum_{l=1}^{k}\mul{A}{k-l}A=\mul{A}{k}k$}
\end{align*}
This proves the claim.
\end{proof}
Lemma \ref{bound Sn} follows from the claim by
$$\E[g_n^2,S_n]\le \sum_{i=0}^n\mul{A}{i}q^{-i/2}\leq \sum_{i=0}^\infty\mul{A}{i}q^{-i/2} =  (1-q^{-1/2})^{-A}=:B.$$
\end{proof}

\begin{lemma}[\textbf{Comparison Lemma}]\label{comparison}
Suppose $V$ is a geometrically connected scheme of finite type over $\fq$ with $d:=\dim V\ge 1$.  There exists a constant $B$ which only depends on $V$ such that for any sequence $f_n:S_n\to[0,1]$ of class functions of $S_n$, we have
$$\E[f_n, \cvq]\leq B\sqrt{\E[f_n, S_n]}$$
\end{lemma}
\begin{proof}
\begin{align}
\nonumber\E[f_n, \cvq]&=\frac{1}{|\cvq|}\sum_{C\in \cvq}f_n(C)\\
\label{ex} &=\frac{1}{|\cvq|}\sum_{\mu\vdash n}\bigg|\Big\{C\in\cvq : \mu_C=\mu\Big\}\bigg|f_n(\mu)
\end{align}
See Remark \ref{notation} for the definition of $\mu_C$. For any partition $\mu=1^{\mu_1}2^{\mu_2}\cdots n^{\mu_n}$ of $n$, notice that
\begin{align*}
\bigg|\Big\{C\in\cvq : \mu_C=\mu\Big\}\bigg|&=\prod_{j=1}^n {\mj\choose \mu_j}\\
\frac{1}{n!}\bigg|\Big\{\s\in S_n : \mu_\s = \mu\Big\}\bigg| & = \bigg(\prod_{j=1}^n j^{\mu_j}\mu_j!\bigg)^{-1}=:z_\mu^{-1}
\end{align*}
Then (\ref{ex}) gives:
\begin{align}
\nonumber\E[f_n,\cvq]&=\frac{q^{dn}}{|\cvq|}\sum_{\mu\vdash n}\bigg[\prod_{j=1}^n {\mj\choose \mu_j}\frac{j^{\mu_j}\mu_j!}{q^{dj\mu_j}}\bigg]\frac{f_n(\mu)}{z_\mu}\\
\label{ex2}&=\frac{q^{dn}}{|\cvq|}\sum_{\mu\vdash n}g_n(\mu)\frac{f_n(\mu)}{z_\mu}&\text{$g_n$ was defined in Lemma \ref{bound Sn}}
\end{align}
By Lemma \ref{sf}, there is a constant $A$ depending only on $V$ such that 
$$\frac{q^{dn}}{|\cvq|}\leq A, \ \ \ \ \ \ \ \forall n.$$
Thus (\ref{ex2}) gives 
\begin{align*}
\E[f_n,\cvq]&\leq A\cdot\sum_{\mu\vdash n}\frac{g_n(\mu)f_n(\mu)}{z_\mu}=A\cdot \frac{1}{n!}\sum_{\s\in S_n}g_n(\s)f_n(\s)\\
&= A\cdot\E[g_n\cdot f_n,S_n]\\
&\leq A\cdot\E[g_n^2,S_n]^{1/2}\cdot\E[f_n^2,S_n]^{1/2}&\text{by Cauchy-Schwarz inequality}\\
&\leq B\cdot\E[f_n^2,S_n]^{1/2}&\text{for another constant $B$ by Lemma \ref{bound Sn}}\\
&\leq B\cdot\E[f_n,S_n]^{1/2}&\text{since $f_n\leq 1$}.
\end{align*}
Lemma \ref{comparison} is proved. 
\end{proof}
\begin{rmk}[\textbf{Comparing measures on $S_n$}]
The set of conjugacy classes in $S_n$ admits a standard probability measure, namely, the uniform measure on $S_n$. 
 Recall in (\ref{pushforward}) we defined the following map for each $n$:
\begin{align*}
\psi_n:\sq&\to \{\text{conjugacy classes in $S_n$}\},
\end{align*}
by $C\mapsto \s_C$.
Thus, one can get nonstandard measures on the set of conjugacy classes of $S_n$ by the pushforward of the uniform measure on $\sq$, or on $\cq\subset \sq$, via $\psi_n$. The proof of Lemma \ref{comparison} is a direct comparison of these measures. In the case when $V=\A^1$, the pushforward of the uniform measure on $\conf^n\A^1(\fq)$ via $\psi_n$ is precisely what Hyde-Lagarias \cite{HL} called \emph{the $q$-splitting measure} on the set of conjugacy classes of $S_n$. 
\end{rmk}

\noindent\textbf{Step 2: Approximate the left hand side of (\ref{bound}) in Theorem \ref{Poisson}.}
We now introduce the following  function which will approximate the left hand side of (\ref{bound}) in Theorem \ref{Poisson}: for $M$ and $N$ with $0<M<N<n$, define
$$\mu_{r,L}(C):= \frac{1}{\log(N)-\log(M)}\cdot \ld\bigg(\big\{t\in[\log M,\log N] : \Big|\phi(C)\cap [t,t+L]\Big|=k\big\}\bigg)$$
where again $\ld$ denotes the Lebesgue measure. 
We point out that $\mu_{r,L}$ also depends on $M,N$, even though we surpress this in our notation. 
Our goal in this step is to prove Proposition \ref{sym est} below, which bounds the expected value of $\mm(C)$ for $C\in \sq$ in certain range of related constants.

\begin{prop}\label{conf est}
There exists a constant $B$ depending only on $V$ such that for any $N, M$ satisfying $M\leq\sqrt{n}\le N\le n$,  for $m:=(10\log\log N)/(\log\log\log N)^2$, and for any $r,L$ with $r\le m/10$, and $(\log\log N)/M<L\le m/10$,
$$\frac{1}{|\cvq|}\sum_{C\in\cq}|\mu_{r,L}(C)-\ee|\leq B\cdot \bigg(\ee\frac{1}{2^{m}}\bigg)^{1/2}.$$

\end{prop}
\begin{proof}
Granville proved in equation (4.1) of \cite{G3} that for any $N, M, m,r,L$ satisfying the assumptions of this proposition,
$$\frac{1}{|S_n|}\sum_{\s\in S_n}|\mu_{r,L}(\s)-\ee|\leq  \ee\frac{1}{2^{m}}.$$
Proposition \ref{conf est} follows by applying  Lemma \ref{comparison} to Granville's estimate.
\end{proof}

\begin{prop}
\label{sym est}
There exists a constant $B$ depending only on $V$ such that for all large enough $n$, for every $\ep>0$, and every $M,N$ satisfying $0\le M \leq\sqrt{n^{1-\ep}}$ and $\sqrt n\le N\le n^{1-\ep}$, let $m:=(10\log\log N)/(\log\log\log N)^2$, for every $r, L$ satisfying $r\le m/10$, and $(\log\log N)/M<L\le m/10$, we have
$$\frac{1}{|\sq|}\sum_{C\in\sq}\bigg|\mu_{r,L}(C)-\ee\bigg|\leq B\cdot \bigg(\ee\frac{1}{2^{m}}\bigg)^{1/2}.$$
\end{prop}
\begin{proof}[Proof of Proposition \ref{sym est}]
Every 0-cycle can be written uniquely as a sum $C+2D$ where $C$ is square-free. This gives the following partition of $\sq$:
$$\sq = \bigcup_{j=0}^{\lfloor{n/2}\rfloor} \conf^{n-2j}V(\fq)\times \sym^{j}V(\fq)$$
Therefore, 
\begin{align}
\nonumber&\frac{1}{|\sq|}\sum_{C\in\sq}\bigg|\mu_{r,L}(C)-\ee\bigg|\\
\label{sq1}&=\frac{1}{|\sq|}\sum_{j=0}^{\lfloor{n/2}\rfloor}\sum_{\substack{C\in\conf^{n-2j}V(\fq) \\ D\in\sym^{j}V(\fq)}}\bigg|\mu_{r,L}(C+2D)-\ee\bigg|&
\end{align}

\begin{lemma}
\label{add}
For any $C\in \sq$ and any $D\in \sym^jV(\fq)$, we have
$$|\mm(C+D)-\mm(C)|\leq \frac{\log j}{\log (N/M)}.$$
\end{lemma}
\begin{proof}[Proof of Lemma \ref{add}.]
$$\mm(C+D) = \frac{1}{\log(N/M)}\ld\bigg(\{t\in [\log M,\log N]: \phi(C+D)\cap [t,t+L]=r\}\bigg)$$
Notice that $\phi(C+D)=\phi(C)\cup\phi(D)$, where $\phi(D)\subset [0,\log j]$. Therefore, $\phi(C+D)$ is the same as $\phi(C)$ except on the small interval $[0,\log j]$. Thus, we have
$$\Bigg|\ld\bigg(\Big\{t\in [\log M,\log N]: \Big|\phi(C+D)\cap [t,t+L]\Big|=r\Big\}\bigg)-\ld\bigg(\Big\{t\in [\log M,\log N]: \Big|\phi(C)\cap [t,t+L]\Big|=r\Big\}\bigg)\Bigg|\leq \log j$$
which gives that 
$$|\mm(C+D)-\mm(C)|\leq \frac{\log j}{\log (N/M)}.$$
\end{proof}
By Lemmas \ref{sf} and \ref{es all}, we know there exists a constant $A$ such that for all $n,j$
$$\frac{|\conf^{n-2j}V(\fq)\times\sym^{j}V(\fq)|}{|\sq|}\leq Aq^{-jd}.$$
Thus, (\ref{sq1}) becomes
\begin{align}
\nonumber&\frac{1}{|\sq|}\sum_{C\in\sq}\bigg|\mu_{r,L}(C)-\ee\bigg|\\
\nonumber&\le A\cdot \sum_{j=0}^{\lfloor{n/2}\rfloor}\Bigg[\frac{1}{|\conf^{n-2j}V(\fq)\times\sym^{j}V(\fq)|}\sum_{\substack{C\in\conf^{n-2j}V(\fq) \\ D\in\sym^{j}V(\fq)}}\bigg|\mu_{r,L}(C+2D)-\ee\bigg|\Bigg]q^{-jd}
\end{align}
Applying Lemma \ref{add} and triangle inequality, we obtain:
\begin{align}
\nonumber&\le A\cdot \sum_{j=0}^{\lfloor{n/2}\rfloor}\Bigg[\frac{1}{|\conf^{n-2j}V(\fq)\times\sym^{j}V(\fq)|}\sum_{\substack{C\in\conf^{n-2j}V(\fq) \\ D\in\sym^{j}V(\fq)}}\bigg|\mu_{r,L}(C)-\ee\bigg|+\frac{\log 2j}{\log(N/M)}\Bigg]q^{-jd}\\
&=\frac{A}{\log(N/M)}\sum_{j=0}^{\lfloor{n/2}\rfloor}\log (2j)q^{-jd}+A\cdot \sum_{j=0}^{\lfloor{n/2}\rfloor}\Bigg[\frac{1}{|\conf^{n-2j}V(\fq)|}\sum_{{C\in\conf^{n-2j}V(\fq) }}\bigg|\mu_{r,L}(C)-\ee\bigg|\Bigg]q^{-jd}\label{sq2}
\end{align}

For the first summand in (\ref{sq2}), we know that the infinite series $\sum_{j=0}^\infty \log(2j)q^{-jd}$ converges by the ratio test. Thus, there exists a constant $B$ such that
\begin{equation}
\label{sq3}
\frac{A}{\log(N/M)}\sum_{j=0}^{\lfloor{n/2}\rfloor}\log (2j)q^{-jd}\leq \frac{B}{\log(N/M)}\ll\bigg(\ee\frac{1}{2^{m}}\bigg)^{1/2}
\end{equation}
where the last equality holds in the prescribed ranges of $M,N,L,r,m$.

For the second summand in (\ref{sq2}), we divide the sum over $j$ into two parts: when $j\leq (n-n^{1-\ep})/2$, and when $j> (n-n^{1-\ep})/2$. When $j\leq (n-n^{1-\ep})/2$, Proposition \ref{conf est} holds for $\conf^{n-2j}V(\fq)$, 
 and thus there is a constant $A$ depending only on $V$ such that
\begin{align*}
\frac{1}{|\conf^{n-2j}V(\fq)|}\sum_{{C\in\conf^{n-2j}V(\fq) }}\bigg|\mu_{r,L}(C)-\ee\bigg|&\le A\cdot \bigg(\ee\frac{1}{2^{m}}\bigg)^{1/2}.
\end{align*}
Thus, we have
\begin{align} 
\nonumber\sum_{j=0}^{\lfloor{n/2}\rfloor}&\Bigg[\frac{1}{|\conf^{n-2j}V(\fq)|}\sum_{{C\in\conf^{n-2j}V(\fq) }}\bigg|\mu_{r,L}(C)-\ee\bigg|\Bigg]q^{-jd}\\
\nonumber&=\sum_{j\leq (n-n^{1-\ep})/2}\Bigg[\frac{1}{|\conf^{n-2j}V(\fq)|}\sum_{{C\in\conf^{n-2j}V(\fq) }}\bigg|\mu_{r,L}(C)-\ee\bigg|\Bigg]q^{-jd}\\
\nonumber&+\sum_{j> (n-n^{1-\ep})/2}\Bigg[\frac{1}{|\conf^{n-2j}V(\fq)|}\sum_{{C\in\conf^{n-2j}V(\fq) }}\bigg|\mu_{r,L}(C)-\ee\bigg|\Bigg]q^{-jd}\\
\nonumber&\le \sum_{j\leq (n-n^{1-\ep})/2}  A\cdot \bigg(\ee\frac{1}{2^{m}}\bigg)^{1/2}q^{-jd}+\sum_{j> (n-n^{1-\ep})/2}q^{-jd}\\
&\nonumber\le A^\prime\cdot \bigg(\ee\frac{1}{2^{m}}\bigg)^{1/2} +A^{\prime\prime}\cdot q^{-d(n-n^{1-\ep})/2}\ \ \  \ \ \ \ \ \text{for some constants $A^\prime$ and $A^{\prime\prime}$}\\
\label{sq4}&\le  A^{\prime\prime\prime}\cdot \bigg(\ee\frac{1}{2^{m}}\bigg)^{1/2}\ \ \ \ \ \ \ \ \ \ \ \ \ \ \  \ \ \ \ \ \ \ \ \ \ \ \ \ \ \text{for some constant $A^{\prime\prime\prime}$}
\end{align}
The last inequality comes from the observation that $q^{-d(n-n^{1-\ep})/2}< q^{-\ep n}\ll \bigg(\ee\frac{1}{2^{m}}\bigg)^{1/2}$ in the given ranges of $r,L,m$. 

Combining (\ref{sq2}), (\ref{sq3}) and (\ref{sq4}), there is a constant $B$ such that
\begin{align}
&\frac{1}{|\sq|}\sum_{C\in\sq}\bigg|\mu_{r,L}(C)-\ee\bigg|\leq B\cdot \bigg(\ee\frac{1}{2^{m}}\bigg)^{1/2}.
\end{align}
\end{proof}

\noindent\textbf{Step 3: Deduce Theorem \ref{Poisson} from Proposition \ref{sym est}. }
The deduction below mimics the deduction of Theorem 1.3 from Proposition 4.2 in \cite{Rh}.

We apply Proposition \ref{sym est} by choosing $\ep:=(\log\log n)^2/\log n$, and $N:=n^{1-\ep}$, and $M:=n^\ep$. Let $m$ be as in Proposition \ref{sym est}. With these choices, any $r,L$ satisfying the assumptions of Theorem \ref{Poisson} will also satisfy the assumptions of Proposition \ref{sym est}.

We  construct the subset $\Sigma\subset \sq$ as follows. 
For nonnegative integers $j,r$, we set $L_j:=m^{-1}(1+2^{-m/6})^j$ and let $\Sigma_{j,r}$ contain all $C\in \sq$ such that
\begin{equation}
\label{Sigma}
\bigg|\mu_{r,L_j}(C)-e^{-L_j}\frac{L_j^r}{r!}\bigg|\ge\frac{1}{2^{m/6+1}} \bigg(e^{-L_j}\frac{L_j^r}{r!} \bigg)^{1/2}.
\end{equation}
We set $\Sigma:=\bigcup_{j,r}\Sigma_{j,r}$ for $0\le j\le 2^{m/6+1}\log m$ and $r\le m(\log m)^{-2}$. The total number of such pairs $(j,r)$ are 
$$2^{m/6+1}\log m\cdot m(\log m)^{-2}=\frac{m}{\log m} 2^{m/6+1}.$$
For each pair $j,r$, by Proposition \ref{sym est} we have
$$|\Sigma_{j,r}|=|\sq|\ O({2^{-m/3}}).$$ 
Thus, 
$$|\Sigma|= |\sq|\ O(\frac{m}{\log m}{2^{-m/6}}),$$
which is within the claimed error. 

Next, we show that all $C\in \sq\setminus \Sigma$ satisfy (\ref{bound}). For any $L\in [1/m,m/20]$,  there is some $j$ such that $L_j\leq L<L_{j+1}$, 
\begin{align*}
\mm(C)&=\sum_{i\leq r}\mu_{i,L}(C)-\sum_{i\leq r-1}\mu_{i,L}(C)\\
&\leq \sum_{i\leq r}\mu_{i,L_j}(C)-\sum_{i\leq r-1}\mu_{i,L_{j+1}}(C)&\text{since $L_j\leq L<L_{j+1}$}\\
&\leq \sum_{i\le r} \bigg[e^{-L_j}\frac{L_j^i}{i!}+ \frac{1}{2^{m/6+1}} \bigg(e^{-L_j}\frac{L_j^i}{i!} \bigg)^{1/2}\bigg]&\text{by (\ref{Sigma})}\\
&\ \ \ \ -\sum_{i\le r-1} \bigg[e^{-L_{j+1}}\frac{L_{j+1}^i}{i!}- \frac{1}{2^{m/6+1}} \bigg(e^{-L_{j+1}}\frac{L_{j+1}^i}{i!} \bigg)^{1/2}\bigg]\\
&= e^{-L_{j+1}}\frac{L_{j+1}^r}{r!}+O\bigg(\frac{m}{2^{m/6}}\underbrace{\sum_{i\leq r}\Big(e^{-L_{j+1}}\frac{L_{j+1}^i}{i!}\Big)^{1/2}}_{(*)}\bigg)
&\text{since $e^{-L_{j}}\frac{L_{j}^i}{i!}= e^{-L_{j+1}}\frac{L_{j+1}^i}{i!}\bigg(1+O(\frac{m}{2^{m/6}})\bigg)$}
\end{align*}
Notice that 
\begin{align*}
(*)&\le \sum_{i=1}^\infty(e^{-L_{j+1}}\frac{L_{j+1}^i}{i!})^{1/2}  \ \ \ \ \ \ \text{the infinite series converges by the ratio test.}
\end{align*}
Thus, we have
$$\mu_{r,L}(C)\le e^{-L}\frac{L^r}{r!}+O\bigg(\frac{m}{2^{m/6}}\bigg)\le e^{-L}\frac{L^r}{r!}\bigg[1+O\bigg(\frac{1}{2^{m/15}}\bigg)\bigg],$$
where the last inequality holds because in the given range of $r$ and $L$, we have
$m2^{-m/10}\ll e^{-L}L^r/r!$.

The lower bound $$\mm(C)\ge e^{-L}\frac{L^r}{r!}\bigg[1+O\bigg(\frac{1}{2^{m/15}}\bigg)\bigg]$$ can be obtained in a similar way by considering 
$$\mm(C)\ge  \sum_{i\leq r} \mu_{i,L_{j+1}}(C)-\sum_{i\leq r-1}\mu_{i,L_j}(C).$$

Since we chose $m$ and $y$ such that $m\le y$, combining the lower and the upper bounds, we have 
$$\mm(C)= e^{-L}\frac{L^r}{r!}\bigg[1+O\bigg(\frac{1}{2^{m/15}}\bigg)\bigg]=e^{-L}\frac{L^r}{r!}\bigg[1+O\bigg(\frac{1}{2^{y/15}}\bigg)\bigg].$$
Finally, to establish (\ref{bound}), we note that
$$\frac{1}{\log n}\ld\bigg(\Big\{t\in[0,\log n] : \Big|\phi(C)\cap [t,t+L]=k\Big|\Big\}\bigg)=\mm(C)+O(\ep).$$
The $\ep$ we have chosen is so small that $$\ep\ll \ee\frac{1}{2^{y/15}}$$
in the given ranges of $r$ and $L$.  Theorem \ref{Poisson} is thus proved. 
\end{proof}

\section{Further questions}
In Table \ref{table}, we see that integer multiplication corresponds to the formal addition of  0-cycles. Is there any operation on 0-cycles that corresponds to integer addition? Can one do additive number theory for 0-cycles? For example, how would a twin prime conjecture for 0-cycles state?

\end{document}